\documentclass{amsart}
\usepackage{amssymb}
\usepackage[mathscr]{eucal}
\usepackage{mystyle}
\begin{document}

\title{Linear recurrence relations in $Q$-systems and difference $L$-operators}
\author{Chul-hee Lee}
\address{School of Mathematics and Physics, The University of Queensland, Brisbane QLD 4072, Australia}
\email{c.lee1@uq.edu.au}
\date{December 2014}
\thanks{This work was supported by the National Research Foundation of Korea (grant \# 2014-021261) and the Australian Research Council.}
\begin{abstract}
We study linear recurrence relations in the character solutions of $Q$-systems obtained from the Kirillov-Reshetikhin modules. We explain how known results on difference $L$-operators lead to a uniform construction of linear recurrences in many examples, and formulate certain conjectural properties predicted in general by this construcion.
\end{abstract}
\maketitle
\section{Introduction}
In this paper we study linear recurrence relations in $Q$-systems, more specifically, those relations among the characters of the Kirillov-Reshetikhin modules. We use results on difference $L$-operators \cite{MR1890924} to construct them, extending previous results \cite{dk2010q,dfk,2009arXiv0905.3776N} uniformly across all classical types. This construction predicts certain structural properties of these recurrences, formulated in Conjecture \ref{mainconj2}, which we are able to check experimentally for more general examples.

Let $\mathfrak{g}$ be a complex simple Lie algebra of rank $r$. Let $I=\{1,\cdots, r\}$ and $\{\alpha_a|a\in I\}$ be the set of simple roots. See Table \ref{Dynkinlab} for their enumeration. Let $t_a=2/(\alpha_a,\alpha_a)\in \{1,2,3\}$ for each $a\in I$, where $\theta$ is the highest root and $(\cdot,\cdot)$ denotes the standard bilinear form on the root lattice $Q=\oplus_{i\in I}\mathbb{Z}\alpha_i$ normalized by the condition $(\theta,\theta)=2$. Let $C$ be the Cartan matrix with entries $C_{ab}=(\alpha_a^{\vee},\alpha_b)$ where $\alpha_a^{\vee}=t_a\alpha_a$. Let us denote the fundamental weights by $\omega_a$ for each $a\in I$. Let $P=\oplus_{i\in I}\mathbb{Z}\omega_i$ be the weight lattice. 
We put $y_a=e^{\omega_a}$ and denote the ring of Laurent polynomials in $\{y_{a}|a\in I\}$ by $\mathbb{Z}[y^{\pm}_{a}]_{a\in I}$ or $\mathbb{Z}[P]$.
For a representation $V$ of $\mathfrak{g}$, we will denote its character by $\chi(V)\in \mathbb{Z}[y^{\pm}_{a}]_{a\in I}$, which is actually an element of $\mathbb{Z}[P]^{W}$ where $W$ denotes the Weyl group.
An irreducible representation of highest weight $\lambda$ will be denoted by $L(\lambda)$. Let $q\in \mathbb{C}^{\times}$ be not a root of unity. We shall use the same notation as above for $U_q(\mathfrak{g})$.

For each $a \in I$,  $m\in \mathbb{Z}_{\ge 0}$, and $u\in \mathbb{C}$, there exists a finite-dimensional irreducible $U_{q}(\hat{\mathfrak{g}})$-module $W^{(a)}_{m}(u)$ called the Kirillov-Reshetikhin module. We get a finite-dimensional $U_{q}(\mathfrak{g})$-module ${\operatorname{res}}\, W^{(a)}_m(u)$ by restriction. As we can ignore the dependence on $u$ as a $U_{q}(\mathfrak{g})$-module, we will simply write it as ${\operatorname{res}}\, W^{(a)}_m$ throughout the paper. Let $Q_m^{(a)}=\chi(\operatorname{res} W_m^{(a)})$ for each $a \in I$ and $m\in \mathbb{Z}_{\ge 0}$. We sometimes write $Q_1^{(a)}$ as $q_a$.

We know that $\{Q^{(a)}_m|a\in I,m\in \mathbb{Z}_{\ge 0}\}$ satisfy the following system of equations
\begin{equation}
(Q^{(a)}_m)^2 = Q^{(a)}_{m+1}Q^{(a)}_{m-1} + 
\prod_{b : C_{ab}\neq 0} \prod_{k=0}^{-C_{a b}-1}
Q^{(b)}_{\lfloor\frac{C_{b a}m - k}{C_{a b}}\rfloor}\quad a\in I, m \ge 1
\label{Qsys}
\end{equation}
where $\lfloor \cdot \rfloor$ denotes the floor function. 
This was conjectured in \cite{Kirillov1990} and proved in \cite{MR1993360,MR2254805}. We call (\ref{Qsys}) the unrestricted $Q$-system of type $\mathfrak{g}$.

\begin{table}\label{Dynkinlab}
\begin{align*}
A_l &&& 
\begin{tikzpicture}[start chain]
\dnode{1}
\dnode{2}
\dydots
\dnode{l-1}
\dnode{l}
\end{tikzpicture}
\\
B_l &&&
\begin{tikzpicture}[start chain]
\dnode{1}
\dnode{2}
\dydots
\dnode{l-1}
\dnodenj{l}
\path (chain-4) -- node[anchor=mid] {\(\Rightarrow\)} (chain-5);
\end{tikzpicture}
\\
C_l &&&
\begin{tikzpicture}[start chain]
\dnode{1}
\dnode{2}
\dydots
\dnode{l}
\dnodenj{l}
\path (chain-4) -- node[anchor=mid] {\(\Leftarrow\)} (chain-5);
\end{tikzpicture}
\\
D_l &&&
\begin{tikzpicture}
\begin{scope}[start chain]
\dnode{1}
\dnode{2}
\node[chj,draw=none] {\dots};
\dnode{l-2}
\dnode{l-1}
\end{scope}
\begin{scope}[start chain=br going above]
\chainin(chain-4);
\dnodebr{l}
\end{scope}
\end{tikzpicture}
\\
E_6 &&&
\begin{tikzpicture}
\begin{scope}[start chain]
\foreach \dyni in {1,...,5} {
\dnode{\dyni}
}
\end{scope}
\begin{scope}[start chain=br going above]
\chainin (chain-3);
\dnodebr{6}
\end{scope}
\end{tikzpicture}
\\
E_7 &&&
\begin{tikzpicture}
\begin{scope}[start chain]
\foreach \dyni in {1,...,6} {
\dnode{\dyni}
}
\end{scope}
\begin{scope}[start chain=br going above]
\chainin (chain-3);
\dnodebr{7}
\end{scope}
\end{tikzpicture}
\\
E_8 &&&
\begin{tikzpicture}
\begin{scope}[start chain]
\foreach \dyni in {1,...,7} {
\dnode{\dyni}
}
\end{scope}
\begin{scope}[start chain=br going above]
\chainin (chain-3);
\dnodebr{8}
\end{scope}
\end{tikzpicture}
\\
F_4 &&&
\begin{tikzpicture}[start chain]
\dnode{1}
\dnode{2}
\dnodenj{3}
\dnode{4}
\path (chain-2) -- node[anchor=mid] {\(\Rightarrow\)} (chain-3);
\end{tikzpicture}
\\
G_2 &&&
\begin{tikzpicture}[start chain]
\dnodenj{1}
\dnodenj{2}
\path (chain-1) -- node {\(\Rrightarrow\)} (chain-2);
\end{tikzpicture}
\end{align*}
\caption{Dynkin diagrams and the enumeration of simple roots}
\end{table}

\subsection*{Statement of main conjectures}
Our main interest is the following conjecture on linear recurrence relations in the sequence $\{Q_m^{(a)}\}_{m\geq 0}$ for fixed $a\in I$.
\begin{conjecture}\label{mainconj}
For each $a\in I$, the sequence $\{Q_m^{(a)}\}_{m\geq 0}$ satisfies a linear recurrence relation. More precisely, there exist a positive integer $\ell_a$ and $C_k^{(a)}\in \mathbb{Z}[y^{\pm}_{a}]_{a\in I},\, k=0,\cdots, \ell_a$ such that the following relation holds
\begin{equation}
\sum_{k=0}^{\ell_a}(-1)^k C_k^{(a)} Q_{n-k}^{(a)}=0\label{CQ}
\end{equation}
for any sufficiently big $n>0$. Here $C_0^{(a)}=1$ and $C_{\ell_a}^{(a)}=\pm 1$. Moreover, $C_k^{(a)}\in\mathbb{Z}[q_1,\cdots, q_r],\, k=0,\cdots, \ell_a$.
\end{conjecture}
To avoid ambiguity, we will fix $\ell_a$ as the minimal one among those with the same property. 
The connection we explore between linear recurrences and $L$-operators leads to the following structural properties of the coefficients $C_k^{(a)}$ in (\ref{CQ}). Let $D$ be an indeterminate.
\begin{conjecture}\label{mainconj2}
There exist finite sets $\Lambda_a$ and $\Lambda'_a$ contained in $P$ such that
\begin{equation}
\sum_{k=0}^{\ell_a}(-1)^k C_{k}^{(a)}D^k=\prod_{\lambda \in \Lambda_a}(1-e^{\lambda}D)\prod_{\lambda \in \Lambda'_a}(1-e^{\lambda}D^{t_a})\label{factor}
\end{equation}
and $\omega_a\in \Lambda_a$. When $t_a=1$, we put $\Lambda_a'=\emptyset$.
\end{conjecture}

An effective approach to both conjectures can be found in the study of difference $L$-operators and $TT$-relations. See, for example, \cite{frenkel1996quantum}, \cite{MR1890924} and references therein. So far the $L$-operators have been constructed only for certain $\mathfrak{g}$ and $a\in I$ although our work indicates that there are more to be studied. We use them to settle Conjectures \ref{mainconj} and \ref{mainconj2} in the corresponding cases. One can sometimes deduce (\ref{CQ}) from a $TT$-relation by restriction in a straightforward way. We also note that (\ref{factor}) is closely related to the existence of the factorized form of a difference $L$-operator.

In the literature there are known results about (\ref{CQ}) in some special cases. The $Q$-systems of type $A$ are considered from the viewpoint of integrable systems in \cite{dk2010q,dfk}. The $Q$-systems of type $A$ and $D$ are studied in \cite{2009arXiv0905.3776N} and the expressions for the coefficients of linear recurrence relations are explicitly given. However, these works do not indicate the relevance of their results on linear recurrence relations with earlier ones on difference $L$-operators and $TT$-relations.

We would like to mention some additional works in relation to our conjectures. 
This work was partly motivated to understand understand certain periodicity phenomenon associated $Q$-systems studied in \cite{lee2013positivity}. A gauge theoretic discussion to linear recurrences among $Q$-systems is given in \cite{MR3279992}. In \cite{2013arXiv1310.6624W}, a uniform construction of conserved quantities of $Q$-systems, which sometimes can be used to derive linear recurrences, is presented. Linear recurrence relations in frieze sequences \cite{MR2729004, MR2824571} have been also studied. Both topics fit into the topic of linear recurrence relations in sequences of cluster transformations.

Conjectures \ref{mainconj} and \ref{mainconj2} pose the following problems :
\begin{itemize}
\item to describe $\Lambda_a$ and $\Lambda_a'$,
\item to determine $\ell_a$,
\item to find the expression for $C_k^{(a)},\,k=0,\cdots, \ell_a$,
\item and to find the generating function
\begin{equation}
\sum_{m=0}^{\infty}Q_m^{(a)}D^m \label{Qgf}
\end{equation}
which is expected to be rational in $D$.
\end{itemize}
These are, of course, closely related with each other. By comparing the coefficients of the first degree term in (\ref{factor}), we obtain 
\begin{equation}
C_1^{(a)}=\sum_{\lambda\in \Lambda_a}e^{\lambda}\label{CLamb}.
\end{equation}
We also deduce
\begin{equation}
\ell_a=|\Lambda_a|+t_a|\Lambda_a'| \label{ellLamb}
\end{equation}
by comparing the degrees of both sides of (\ref{factor}).
From the general theory of linear recurrence relations we also know that the numerator of the rational function (\ref{Qgf}) is given by (\ref{factor}).

Experiments have shown that the difference between $\ell_a$ and $\dim W^{(a)}_1$ is relatively small in many cases and (\ref{ellLamb}) partly clarifies this. For example, if $C_1^{(a)}=Q_{1}^{(a)}+\delta_a$ for some integer $\delta_a$ (which is true in many cases), then $|\Lambda_a|=\dim W^{(a)}_1+\delta_a$ from (\ref{CLamb}). If we further assume that $t_a=1$, then we can conclude from (\ref{ellLamb}) that 
\begin{equation}
\ell_a=|\Lambda_a|=\dim W^{(a)}_1+\delta_a. \label{elldim}
\end{equation}

Since this work is primarily of experimental nature we briefly explain how one can carry out experiments to get evidences for our conjectures. 
Note that a choice of complex numbers for each $Q_1^{(a)},\, a\in I$ gives rise to a homomorphism 
$$\varphi:\mathbb{Z}[P]^{W}\to \mathbb{C}.$$ We then use (\ref{Qsys}) in order to get $\varphi(Q_m^{(a)})$. As the expression of $Q_m^{(a)}$ gets complicated quite fast as $m$ grows, it is more convenient to work with the sequence $\{\varphi(Q_m^{(1)})\}_{m\geq 0}$ of complex numbers obtained in this way. Then we try to find a conjectural linear recurrence relation in it and for this purpose we have used {\it Mathematica} command \verb+FindLinearRecurrence+, available in {\it Mathematica} version 7.0 or higher.

\begin{example}\label{e6ex}
Let us consider a concrete example in type $E_6$. If we use the following randomly chosen initial condition 
$$
\begin{array}{|c|c|c|c|c|c|}
 q_1 & q_2 & q_3 & q_4 & q_5 & q_6 \\
\hline
 17 & 22 & 38 & 40 & 14 & 31
\end{array},
$$
the corresponding solution of the $Q$-system of type $E_6$ is given by
$$
\begin{array}{c|c|c|c|c|c|c}
m\backslash a& 1 & 2 & 3 & 4 & 5 & 6 \\
\hline
0&  1 & 1 & 1 & 1 & 1 & 1 \\ 
1& 17 & 22 & 38 & 40 & 14 & 31 \\ 
2& 267 & -162 & -25836 & 1068 & 156 & 923 \\ 
3& 4203 & 314748 & 21768228 & 129276 & 1662 & 28315 \\
\vdots & \vdots & \vdots & \vdots & \vdots & \vdots & \vdots
\end{array}
$$
Let $\varphi:\mathbb{Z}[P]^{W}\to \mathbb{C}$ be the corresponding homomorphism and consider the sequence $\{\varphi(Q_m^{(1)})\}_{m\geq 0}$, which goes like : 
$$
1, 17, 267, 4203, 64983, 1015833, 15856320,\cdots.
$$
An experimental observation is that this integer sequence obeys the following linear recurrence relation
$$
\sum_{k=0}^{27}(-1)^k \varphi(C_k^{(1)})\varphi(Q_{n-k}^{(1)})=0
$$
where the coefficients $\varphi(C_k^{(1)})$ are given in Table \ref{tabE6}. By repeating the same kind of experiments with different choices of initial conditions one can come up with the conjectural expression for $C_k^{(1)}\in \mathbb{Z}[q_1,\cdots, q_6]$ as in Table \ref{tabE6}.
\begin{table}
\begin{tabular}{c|c|c}
 $k$ & $\varphi(C_k^{(1)})$ & $C_k^{(1)}$\\
\hline
 0 & 1 & 1 \\
 1 & 17 & $q_1$ \\
 2 & 8 & $q_2-q_5$ \\
 3 & -230 & $q_3-q_1 q_5-q_6+1$ \\
 4 & 422 & $-q_6 q_1+q_1-q_2 q_5+q_4 q_6$ \\
 $\cdots$ & $\cdots$ & $\cdots$ \\
 23 & -418 & $-q_1 q_4+q_5+q_2 q_6-q_5 q_6$ \\
 24 & -230 & $q_3-q_1 q_5-q_6+1$ \\
 25 & 23 & $q_4-q_1$ \\
 26 & 14 & $q_5$ \\
 27 & 1 & 1 \\
\end{tabular}
\caption{\label{tabE6}Coefficients of linear recurrence relations in Example \ref{e6ex}}
\end{table}

We recognize that 27 is the same as the dimension of the smallest nontrivial irreducible representation $L(\omega_1)$ of the simple Lie algebra $E_6$. How can we understand this coincidence? This is an example where (\ref{elldim}) holds with $\delta_a=0$. We will revisit the $E_6$ case in subsection \ref{e6eg} again, where we give a conjectural description of $\Lambda_1$ in (\ref{factor}) and see this coincidence in a clearer way.
\end{example}

This paper is organized as follows. 
In Section \ref{review} we collect known results about difference $L$-operators in classical types and use them to prove Conjectures \ref{mainconj} and \ref{mainconj2} in some cases.
In Section \ref{samples} we give a description of the objects appearing in Conjectures \ref{mainconj} and \ref{mainconj2} when $\mathfrak{g}$ is one of exceptional types, for which there is no difference $L$-operator available. In Appendix we give some tables for $\ell_a$ and also some tables related to the growth of the dimensions of the Kirillov-Reshetikhin modules.

\section{Difference $L$-operators and application to linear recurrence relations}\label{review}
In this section we give a review of results on difference $L$-operators and prove Conjectures \ref{mainconj} and \ref{mainconj2} for the cases where we have the corresponding difference $L$-operators. The proofs are simple and straightforward. Although some of these results are already known, our proofs may give a new perspective on the subject.

A relatively recent survey on the subject is given in Section 9 of \cite{1751-8121-44-10-103001}, on which our discussion is based. The difference $L$-operators in classical types are studied in \cite{frenkel1996quantum}, \cite{kuniba1995quantum}, \cite{tsuboi1996solutions} and \cite{MR1890924}. See \cite{tsuboi2002difference} also for the $L$-operators associated with twisted quantum affine algebras. 

Let $\mathbb{Z}[Y^{\pm}_{a,z}]_{a\in I, z \in \mathbb{C}^{\times}}$ be the ring of Laurent polynomials in $\{Y_{a,z}|a\in I, z \in \mathbb{C}^{\times}\}$. There exists a surjective ring homomorphism
\begin{equation}
\operatorname{res} : \mathbb{Z}[Y^{\pm}_{a,z}]_{a\in I, z \in \mathbb{C}^{\times}} \to \mathbb{Z}[y^{\pm}_{a}]_{a\in I} \label{res}
\end{equation}
given by $Y_{a,z}\mapsto y_a$. For $f:\mathbb{C}\to  \mathbb{Z}[Y^{\pm}_{a,z}]_{a\in I, z \in \mathbb{C}^{\times}}$, let us define the difference operator $D$ by $Df(u): = f(u+2)$ in type $A_r, B_r$ and $D_r$ and $Df(u): = f(u+1)$ in type $C_r$. 

For each $\mathfrak{g}$ of classical type, we will define a finite set $J$ and $z_j(u)\in \mathbb{Z}[Y^{\pm}_{a,z}]_{a\in I, z \in \mathbb{C}^{\times}}$ for each $j\in J$. The set of weights of $\operatorname{res} W_{1}^{(1)}$ is given by $\{\lambda_j|j\in J\}$ where $\lambda_j\in P\, (j\in J)$ satisfies 
$$
\operatorname{res} z_{j}(u)=e^{\lambda_j}\in \mathbb{Z}[y^{\pm}_{a}]_{a\in I}.
$$
The difference $L$-operator $L(u)$ is given as a certain product of the operators of the form $(1\pm z_{j}(u)D)^{\pm 1}$ acting on the elements of $\{f|f:\mathbb{C}\to \mathbb{Z}[Y^{\pm}_{a,z}]_{a\in I, z \in \mathbb{C}^{\times}}\}$. For example, in type $A_r$, we have
$J=\{1,\cdots, r+1\}$ and the $L$-operator $L(u)$, given by
$$
L(u)= (1-z_{r+1}(u)D)\cdots (1-z_2(u)D)(1-z_1(u)D)
$$
where $z_j(u)$ is given in (\ref{zaa}). The order of product is of crucial importance due to the non-commutativity of operations involved here. We shall use the following notation
\begin{equation}
\prod_{1 \le i \le k}^{\longrightarrow}X_i = X_1X_2\cdots X_k,
\qquad
\prod_{1 \le i \le k}^{\longleftarrow}X_i = X_k\cdots X_2X_1.
\end{equation}

There is a notion of $q$-character $\chi_q(V) \in \mathbb{Z}[Y^{\pm}_{a,z}]_{a\in I, z \in \mathbb{C}^{\times}}$ for a finite-dimensional representation $V$ of $U_{q}(\hat{\mathfrak{g}})$, which is introduced in \cite{frenkel1999q}. We will denote $\chi_q\left (W^{(a)}_m(u)\right)$ by $T^{(a)}_m(u)$ in the following subsections. Although we do not discuss the details of it here, we should keep in mind that
$$
\operatorname{res}T_m^{(a)}(u)=Q^{(a)}_m.
$$

\subsection{Type $A_r$}
Let $J=\{1,2\ldots, r+1\}$. For each $j\in J$, let 
\begin{equation}
z_j(u)=Y^{-1}_{j-1, q^{u+j}}Y_{j, q^{u+j-1}}\label{zaa}
\end{equation}
where $Y_{0, q^u}=Y_{r+1, q^u}=1$.

The difference $L$-operator is given by
\begin{equation}
\begin{aligned}
L(u)&= (1-z_{r+1}(u)D)\cdots (1-z_2(u)D)(1-z_1(u)D)\\
	&= \sum_{a=0}^{r+1}(-1)^aT^{(a)}_1(u+a-1)D^a
\label{La}
\end{aligned}
\end{equation}
where $T^{(0)}_1(u)= T^{(r+1)}_1(u)=1$. And the multiplicative inverse of $L(u)$ is
\begin{equation}
\begin{aligned}\label{linva}
L(u)^{-1} & =(1-z_1(u)D)^{-1}(1-z_2(u)D)^{-1}\cdots (1-z_{r+1}(u)D)^{-1} \\
		  & =\sum_{m \ge 0}  T^{(1)}_m(u+m-1) D^{m}.
\end{aligned}
\end{equation}

By multiplying (\ref{La}) and (\ref{linva}) in two different orders, 
we get the following $TT$-relations : 
\begin{equation}\label{TTa}
\begin{aligned}
&\sum_{0 \le a \le \min(r+1,m)}(-1)^aT^{(a)}_1(u+a)
T^{(1)}_{m-a}(u+m+a)=\delta_{m 0},\\
&\sum_{0 \le a \le \min(r+1,m)}(-1)^aT^{(a)}_1(u+m-a)
T^{(1)}_{m-a}(u-a) = \delta_{m 0}
\end{aligned}
\end{equation}
for $m \in \mathbb{Z}_{\geq 0}$.

In this case the following theorem is hardly new. 
\begin{theorem}\label{linA}
Let $\mathfrak{g}$ of type $A_r$. Conjectures \ref{mainconj} and \ref{mainconj2} hold for $a=1$ with
\begin{itemize}
\item $\Lambda_1=\{\lambda_j|j\in J\}$,
\item $\ell_1=r+1$,
\item $C_k^{(1)}=Q_1^{(k)}=\chi\left(\Lambda^k(\operatorname{res} W^{(1)}_1)\right)$ for $k=0,1,\cdots,r+1$.
\end{itemize}
We have
$$
\sum_{m=0}^{\infty}Q_{m}^{(1)}D^m=\frac{1}{\prod_{\lambda \in \Lambda_1}\left(1-e^{\lambda}D\right)}\label{genQa}.
$$
\end{theorem}
\begin{proof}
We first check that
$$\begin{aligned}
\prod_{\lambda \in \Lambda_1}(1-e^{\lambda}D)&=\prod_{j \in J}(1-e^{\lambda_j}D) \\
&=\sum_{k=0}^{r+1}(-1)^k C_{k}^{(1)}D^k.
\end{aligned}
$$

The image of the $L$-operator (\ref{La}) under the restriction map is
$$
\operatorname{res} L(u)=\sum_{k=0}^{r+1}(-1)^k C_{k}^{(1)}D^k.
$$
From (\ref{linva}) we have
$$
\operatorname{res} L(u)^{-1}=\sum_{m=0}^{\infty}Q_{m}^{(1)}D^m=\frac{1}{\prod_{\lambda \in \Lambda_1}\left(1-e^{\lambda}D\right)}.
$$
Finally we see that
$$
\left(\sum_{k=0}^{r+1}(-1)^k C_{k}^{(1)}D^k\right)\left(\sum_{m=0}^{\infty}Q_{m}^{(1)}D^m\right)=1.
$$
It implies (\ref{CQ}), which is nothing but the image of (\ref{TTa}) under the restriction map.
\end{proof}

\begin{remark} 
This has been long known as one can see from (\ref{TTa}). See also (4.3) in \cite{dk2010q}, Theorem 2.8 in \cite{dfk} and Theorem 2.5 in \cite{2009arXiv0905.3776N}.
\end{remark}
Our experiments show that $C_1^{(a)}=Q_{1}^{(a)}$ and $\ell_a=\dim W^{(a)}_1$ for any $a\in I$. This leads us to the following : 
\begin{conjecture}\label{Aconj}
Let $\mathfrak{g}$ of type $A_r$. Conjectures \ref{mainconj} and \ref{mainconj2} hold for any $a\in I$ with
\begin{itemize}
\item $\Lambda_a=\text{the set of weights of }\operatorname{res} W^{(a)}_1$,
\item $\ell_a=\dim W^{(a)}_1$,
\item $C_k^{(a)}=\chi\left(\Lambda^k(\operatorname{res} W^{(a)}_1)\right)$ for $k=0,\cdots,\ell_a$.
\end{itemize}
\end{conjecture}

\subsection{Type $B_r$}
Let $J=\{1,2\ldots, r, 0, \overline{r},\ldots, \overline{2},\overline{1}\}$. For $j \in J$, we define $z_j(u)$ as
$$
\begin{aligned}
& z_a(u)= Y_{a, q^{2u+2a-2}}Y_{a-1, q^{2u+2a}}^{-1}
\qquad (1\le a \le r-1),\\
& z_r(u)= Y_{r, q^{2u+2r-3}}Y_{r, q^{2u+2r-1}}
Y_{r-1, q^{2u+2r}}^{-1},\\
& z_0(u) = Y_{r, q^{2u+2r-1}} Y_{r, q^{2u+2r-3}} 
Y^{-1}_{r, q^{2u+2r+1}} Y^{-1}_{r, q^{2u+2r-1}},\\
& z_{\overline{r}}(u) = Y_{r-1,q^{2u+2r-2}}Y_{r, q^{2u+2r-1}}^{-1}
Y_{r, q^{2u+2r+1}}^{-1},\\
& z_{\overline{a}}(u)= Y_{a-1, q^{2u+4r-2a-2}}
Y_{a, q^{2u+4r-2a}}^{-1}
\qquad (1\le a \le r-1)
\end{aligned}
$$
where $Y_{0, q^u}=1$.

The difference $L$-operator is given by
\begin{equation}
L(u) = \prod_{1 \le a \le r}^{\longrightarrow}
(1-z_{\overline{a}}(u)D)\cdot
(1+z_0(u)D)^{-1}\cdot
\prod_{1 \le a \le r}^{\longleftarrow}(1-z_a(u)D)\label{Lb}
\end{equation}
Note that this is not a polynomial in $D$ as opposed to (\ref{La}). The multiplicative inverse of $L(u)$ is
\begin{equation}
L(u)^{-1} = \sum_{m \ge 0}  T^{(1)}_m(u+m-1) D^{m}. \label{Linvb}
\end{equation}

\begin{theorem}
Let $\mathfrak{g}$ of type $B_r$. Conjectures \ref{mainconj} and \ref{mainconj2} hold for $a=1$ with
\begin{itemize}
\item $\Lambda_1=\{\lambda_j|j\in J\backslash\{0\}\}$,
\item $\ell_1=2r$,
\item $C_k^{(1)}=\sum_{n=0}^{k}(-1)^{k-n} \chi(\Lambda^n \operatorname{res} W_1^{(1)})$ for $k=0,\cdots,2r$.
\end{itemize}
We have
\begin{equation}
\sum_{m=0}^{\infty}Q_{m}^{(1)}D^m=\frac{1+D}{\prod_{\lambda\in \Lambda_1}\left(1-e^{\lambda}D\right)}\label{genQb}.
\end{equation}
\end{theorem}

\begin{proof}
Note that
$$\begin{aligned}
\prod_{\lambda \in \Lambda_1}\left(1-e^{\lambda}D\right) & = \frac{\prod_{j\in J}\left(1-e^{\lambda_j}D\right)}{(1-D)}\\
&=\frac{\sum_{n=0}^{2r+1}(-1)^n \chi(\Lambda^n \operatorname{res} W_1^{(1)}) D^n }{(1-D)} \\
&=\sum_{k=0}^{2r}(-1)^k C_{k}^{(1)}D^k
\end{aligned}
$$
where we have used the relation
\begin{equation*}
\sum_{n=0}^{2r+1}(-1)^{n} \chi(\Lambda^n \operatorname{res} W_1^{(1)})=0
\end{equation*}
in the last line. 

The image of the $L$-operator (\ref{Lb}) under the restriction map is
$$\begin{aligned}
\operatorname{res} L(u)&=\frac{\prod_{j\in J\backslash\{0\}}\left(1-e^{\lambda_j}D\right)}{(1+D)}\\
&=\frac{1}{(1+D)}\sum_{k=0}^{2r}(-1)^k C_{k}^{(1)}D^k.
\end{aligned}
$$
From (\ref{Linvb}) we obtain
$$\operatorname{res} L(u)^{-1}=\sum_{m=0}^{\infty}Q_{m}^{(1)}D^m=\frac{1+D}{\prod_{\lambda\in \Lambda_1}\left(1-e^{\lambda}D\right)}
$$
Finally we get
$$
\frac{1}{(1+D)}\left(\sum_{k=0}^{2r}(-1)^k C_{k}^{(1)}D^k\right)\left(\sum_{m=0}^{\infty}Q_{m}^{(1)}D^m\right)=1
$$
which implies (\ref{CQ}). 
\end{proof}
\begin{remark}
We can easily check that
\begin{itemize}
\item $C_0^{(1)}= 1$,
\item $C_k^{(1)}=q_k-q_{k-1}$ for $k=1,\cdots,r-1$,
\item $C_r^{(1)}=q_r^2-2q_{r-1}$,
\item $C_k^{(1)}=C_{2r-k}^{(1)}$ for $k=r+1,\cdots,2r$.
\end{itemize}
\end{remark}

\subsection{Type $C_r$}
Let $J=\{1,2\ldots, r, \overline{r},\ldots, \overline{2},\overline{1}\}$. For $1 \le a \le r$, we put
$$
\begin{aligned}
& z_a(u)= Y_{a, q^{2u+a-1}}Y_{a-1, q^{2u+a}}^{-1},\\
& z_{\overline{a}}(u)= Y_{a-1, q^{2u+2r-a+2}}
Y_{a, q^{2u+2r-a+3}}^{-1}
\end{aligned}
$$
where $Y_{0, q^u}=1$.

The difference $L$-operator is given by
\begin{equation}
\begin{aligned}
L(u)& = \prod_{1 \le a \le r}^{\longrightarrow}
(1-z_{\overline{a}}(u)D)\cdot
(1-z_{\overline{r}}(u)z_r(u+1)D^2)\cdot
\prod_{1 \le a \le r}^{\longleftarrow}(1-z_a(u)D) \\
& =\sum_{a=0}^r(-1)^a T^{(a)}_1(u+\frac{a-1}{2})D^a-\sum_{a=r+2}^{2r+2}(-1)^a T^{(2r+2-a)}_1(u+\frac{a-1}{2})D^a\label{Lc}
\end{aligned}
\end{equation}
where $T^{(0)}_1(u)=1$. See Theorem 2.5 in \cite{MR1890924}. The multiplicative inverse of $L(u)$ is
\begin{equation}
\begin{aligned}\label{linvc}
L(u)^{-1} & =(1-z_1(u)D)^{-1}(1-z_2(u)D)^{-1}\cdots (1-z_{r+1}(u)D)^{-1} \\
		  & =\sum_{m \ge 0}  T^{(1)}_m(u+\frac{m-1}{2}) D^{m}.
\end{aligned}
\end{equation}\label{TTc}
Again we can derive two $TT$-relations similar to (\ref{TTa}) :
\begin{equation}
\begin{aligned}
\sum_{0 \le a \le \min(2r+2,m)}(-1)^a T^{(a)}_1(u+\frac{m-a}{2}) T^{(1)}_{m-a}(u-\frac{a}{2}) &= \delta_{m 0},\\
\sum_{0 \le a \le \min(2r+2,m)}(-1)^a T^{(a)}_1(u+\frac{a}{2}) T^{(1)}_{m-a}(u+\frac{m+a}{2})&= \delta_{m 0}
\end{aligned}
\end{equation}
for $m \in \mathbb{Z}_{\ge 0}$. Here we put $T_1^{(r+1)}(u)=0$ and $T_1^{(a)}(u):=-T_1^{(2r+2-a)}(u)$ for $a=r+2,\cdots, 2r+2$.

\begin{theorem}
Let $\mathfrak{g}$ of type $C_r$. Conjectures \ref{mainconj} and \ref{mainconj2} hold for $a=1$ with
\begin{itemize}
\item $\Lambda_1=\{\lambda_j|j\in J\}$ and $\Lambda'_1=\{0\}$,
\item $\ell_1=2r+2$,
\item $C_k^{(1)}=\chi(\Lambda^k \operatorname{res} W_1^{(1)})-\chi(\Lambda^{k-2} \operatorname{res} W_1^{(1)})$ for $k=0,\cdots,2r+2$.
\end{itemize}
We have
\begin{equation}
\sum_{m=0}^{\infty}Q_{m}^{(1)}D^m=\frac{1}{(1-D^2)\prod_{\lambda\in \Lambda_1}(1-e^{\lambda}D)}\label{genQc}.
\end{equation}
\end{theorem}

\begin{proof}
We check that
$$\begin{aligned}
\prod_{\lambda \in \Lambda_1}\left(1-e^{\lambda}D\right)\prod_{\lambda \in \Lambda_1'}\left(1-e^{\lambda}D^2\right)&=(1-D^2)\prod_{j\in J}\left(1-e^{\lambda_j}D\right)\\
&=(1-D^2)\sum_{k=0}^{2r}(-1)^k \chi(\Lambda^k \operatorname{res} W_1^{(1)})D^k \\
&=\sum_{k=0}^{2r+2}(-1)^k C_{k}^{(1)}D^k.
\end{aligned}
$$

The image of the $L$-operator (\ref{Lc}) under the restriction map is
$$
\operatorname{res} L(u)=\sum_{k=0}^{2r+2}(-1)^k C_{k}^{(1)}D^k.
$$
From (\ref{linvc}) we have
$$\operatorname{res} L(u)^{-1}=\sum_{m=0}^{\infty}Q_{m}^{(1)}D^m=\frac{1}{(1-D^2)\prod_{\lambda \in \Lambda_1}\left(1-e^{\lambda}D\right)}$$

Finally we obtain
$$
\left(\sum_{k=0}^{2r+2}(-1)^k C_{k}^{(1)}D^k\right)\left(\sum_{m=0}^{\infty}Q_{m}^{(1)}D^m\right)=1.
$$
This implies (\ref{CQ}), which is again the image of (\ref{TTc}) under the restriction map.
\end{proof}
\begin{remark}
From (\ref{Lc}) we know
\begin{itemize}
\item $C_0^{(1)}= 1$,
\item $C_k^{(1)}=q_k$ for $k=1,\cdots,r$,
\item $C_{r+1}^{(1)}=0$,
\item $C_k^{(1)}=-C_{2r+2-k}^{(1)}$ for $k=r+2,\cdots, 2r+2$.
\end{itemize}
\end{remark}

\subsection{Type $D_r$}
Let $J=\{1,2\ldots, r, \overline{r},\ldots, \overline{2},\overline{1}\}$.
For $j \in J$, we define $z_j(u)$ as
$$
\begin{aligned}
& z_a(u)= Y_{a, q^{u+a-1}}Y_{a-1, q^{u+a}}^{-1}
\qquad (1 \le a \le r-2),\\
&z_{r-1}(u) 
= Y_{r-1, q^{u+r-2}}Y_{r, q^{u+r-2}}Y_{r-2,q^{u+r-1}}^{-1},\\
&z_r(u) = Y_{r, q^{u+r-2}}Y_{r-1,q^{u+r}}^{-1},\\
&z_{\overline{r}}(u) = Y_{r-1,q^{u+r-2}}Y_{r,q^{u+r}}^{-1},\\
&z_{\overline{r-1}}(u) = Y_{r-2,q^{u+r-1}}
Y_{r-1,q^{u+r}}^{-1}Y_{r,q^{u+r}}^{-1},\\
& z_{\overline{a}}(u)= Y_{a-1, q^{u+2r-a-2}}
Y_{a, q^{u+2r-a-1}}^{-1}
\qquad (1 \le a \le r-2)
\end{aligned}
$$
where $Y_{0, q^u}=1$.

The difference $L$-operator is given by
\begin{equation}
L(u) = \prod_{1 \le a \le r}^{\longrightarrow}
(1-z_{\overline{a}}(u)D)\cdot
(1-z_r(u)z_{\overline{r}}(u+2)D^2)^{-1}\cdot
\prod_{1 \le a \le r}^{\longleftarrow}(1-z_a(u)D).\label{Ld}
\end{equation}
Again this is not a polynomial in $D$. The multiplicative inverse of $L(u)$ is
\begin{equation}
L(u)^{-1} = \sum_{m \ge 0}  T^{(1)}_m(u+m-1) D^{m}. \label{linvd}
\end{equation}

\begin{theorem}
Let $\mathfrak{g}$ of type $D_r$. Conjectures \ref{mainconj} and \ref{mainconj2} hold for $a=1$ with
\begin{itemize}
\item $\Lambda_1=\{\lambda_j|j\in J\}$,
\item $\ell_1=2r$,
\item $C_k^{(1)}=\chi(\Lambda^k \operatorname{res} W_1^{(1)})$ for $k=0,\cdots,2r$.
\end{itemize}
We have
\begin{equation}
\sum_{m=0}^{\infty}Q_{m}^{(1)}D^m=\frac{1-D^2}{\prod_{\lambda \in \Lambda_1}\left(1-e^{\lambda}D\right)}\label{genQd}.
\end{equation}
\end{theorem}

\begin{proof}
Note that
$$
\begin{aligned}
\prod_{\lambda\in \Lambda_1}\left(1-e^{\lambda}D\right) &=\prod_{j\in J}\left(1-e^{\lambda_j}D\right)\\
&=\sum_{k=0}^{2r}(-1)^k \chi(\Lambda^k \operatorname{res} W_1^{(1)}) D^k \\
&=\sum_{k=0}^{2r}(-1)^k C_k^{(1)} D^k.
\end{aligned}
$$
The image of the $L$-operator (\ref{Ld}) under the restriction map is
$$\begin{aligned}
\operatorname{res} L(u)&=\frac{\prod_{\lambda\in \Lambda_1}\left(1-e^{\lambda}D\right)}{1-D^2}\\
&=\frac{\sum_{k=0}^{2r}(-1)^k C_k^{(1)} D^k}{1-D^2}
\end{aligned}.
$$
From (\ref{linvd}) we obtain
$$\operatorname{res} L(u)^{-1}=\sum_{m=0}^{\infty}Q_{m}^{(1)}D^m=\frac{1-D^2}{\prod_{\lambda\in \Lambda_1}\left(1-e^{\lambda}D\right)}.
$$
Finally we see that
$$
\frac{1}{(1-D^2)}\left(\sum_{k=0}^{2r}(-1)^k C_{k}^{(1)}D^k\right)\left(\sum_{m=0}^{\infty}Q_{m}^{(1)}D^m\right)=1
$$
which implies (\ref{CQ}).
\end{proof}
\begin{remark}
It is a simple exercise to check that
\begin{itemize}
\item $C_0^{(1)}= 1$,
\item $C_1^{(1)}=q_1$,
\item $C_k^{(1)}=q_k-q_{k-2}$ for $k=2,\cdots,r-2$,
\item $C_{r-1}^{(1)}=q_{r-1}q_r-q_{r-3}$, 
\item $C_r^{(1)}=q_{r-1}^2+q_r^2-2q_{r-2}$,
\item $C_k^{(1)}=C_{2r-k}^{(1)}$ for $k=r+1,\cdots,2r$
\end{itemize}
where $q_0=q_{r+1}=1$. The linear recurrence relation (\ref{CQ}) with these coefficients is obtained in Theorem 3.2 of \cite{2009arXiv0905.3776N}.
\end{remark}
By the same reasoning as in Conjecture \ref{Aconj}, we have the following :
\begin{conjecture}
Let $\mathfrak{g}$ of type $D_r$. For $a=r-1$ or $a=r$, Conjectures \ref{mainconj} and \ref{mainconj2} hold with
\begin{itemize}
\item $\Lambda_a=\text{the set of weights of }\operatorname{res} W^{(a)}_1$,
\item $\ell_a=\dim W^{(a)}_1$,
\item $C_k^{(a)}=\chi\left(\Lambda^k(\operatorname{res} W^{(a)}_1)\right)$ for $k=0,\cdots,\ell_a$.
\end{itemize}
\end{conjecture}

\section{Linear recurrence relations in exceptional types}\label{samples}
In this section we provide a conjectural description of $\Lambda_a,\Lambda'_a,\ell_a$ and $C_k^{(a)},\, k=0,\cdots, \ell_a$ in Conjectures \ref{mainconj} and \ref{mainconj2} for some $a\in I$ in exceptional types. We hope that it gives us some hints on how to construct difference $L$-operators in these types if there is any.

\subsection{Type $E_6$}\label{e6eg}
We briefly considered this case in Example \ref{e6ex} for $a=1$. The smallest nontrivial irreducible representation of $E_6$ is $L(\omega_1)$ of dimension 27 and all the weights have multiplicity 1.

For $a=1$, we expect that Conjectures \ref{mainconj} and \ref{mainconj2} hold with
\begin{itemize}
\item $\Lambda_1=\text{the set of weights of } L(\omega_1)$,
\item $\ell_1=27$,
\item $C_k^{(1)}=\chi(\Lambda^k \operatorname{res} W_1^{(1)})$ for $k=0,\cdots, 27$.
\end{itemize}
We also have
\begin{equation}
\sum_{m=0}^{\infty}Q_{m}^{(1)}D^m=\frac{c_0+c_1D+\cdots+c_{15}D^{15}}{\prod_{\lambda \in \Lambda_1}(1-e^{\lambda} D)}\label{gene6}
\end{equation}
where $c_n$ is given as
$$
\begin{array}{c|c||c|c}
 n & c_n & n & c_n\\ \hline
 0 & 1 & 15 & 1\\ \hline
 1 & 0 & 14 & 0\\ \hline
 2 & -\chi \left(L\left(\omega _5\right)\right) & 13 & -\chi \left(L\left(\omega _1\right)\right) \\ \hline
 3 & \chi \left(L\left(\omega _6\right)\right) & 12 & \chi \left(L\left(\omega _6\right)\right) \\ \hline
 4 & 0 & 11 & 0 \\ \hline
 5 & -\chi \left(L\left(\omega _2\right)\right) & 10 & -\chi \left(L\left(\omega _4\right)\right) \\ \hline
 6 & \chi \left(L\left(\omega _1+\omega _5\right)\right) & 9 & \chi \left(L\left(\omega _1+\omega _5\right)\right) \\ \hline
 7 & -\chi \left(L\left(2 \omega _5\right)\right) & 8 & -\chi \left(L\left(2 \omega _1\right)\right)
\end{array}.
$$
As $\operatorname{res}W_{m}^{(1)}=L(m\omega_1)$, (\ref{gene6}) can be understood entirely within the context of the representation theory of $\mathfrak{g}$.

Note that 
$$
\begin{aligned}
&\dim L(\omega_5)=\dim L(\omega_1)&=27, \\
&\dim L(\omega_6)&=78, \\
&\dim L(\omega_2)=\dim L(\omega_4)&=351, \\
&\dim L(\omega_1+\omega_5)&=650, \\
&\dim L(2\omega_5)=\dim L(2\omega_1)&=351.
\end{aligned}
$$
We see that (\ref{gene6}) is consistent with the generating function of dimensions of the Kirillov-Reshetikhin modules :
$$
\begin{aligned}
\sum_{m=0}^{\infty}(\dim W_{m}^{(1)})D^m  &=\frac{1-27 D^2+78 D^3+\cdots+78 D^{12}-27 D^{13}+D^{15}}{(1-D)^{27}}\\
&=\frac{1+10 D+28 D^2+28 D^3+10 D^4+D^5}{(1-D)^{17}}\\
&=1+27 D+351 D^2+3003 D^3+19305 D^4+\cdots.
\end{aligned}
$$

\subsection{Type $E_7$}
The smallest nontrivial irreducible representation of $E_7$ is given by $L(\omega_6)$ of dimension 56 and all the weights have multiplicity 1.

For $a=6$, we expect that Conjectures \ref{mainconj} and \ref{mainconj2} hold with
\begin{itemize}
\item $\Lambda_6=\text{the set of weights of }L(\omega_6)$,
\item $\ell_6=56$,
\item $C_k^{(6)}=\chi(\Lambda^k \operatorname{res} W_1^{(6)})$ for $k=0,\cdots, 56$.
\end{itemize}
In particular, we expect
$$
C_1^{(6)}=\sum_{\lambda \in \Lambda_6}e^{\lambda}=\chi(L(\omega_6))=q_6.
$$

\subsection{Type $E_8$}
The smallest nontrivial irreducible representation of $E_8$ is given by $L(\omega_7)$ of dimension 248. As it is isomorphic to the adjoint representation, we know that the weights are given by 240 distinct roots and $0$ with multiplicity 8.

For $a=7$, we expect that Conjectures \ref{mainconj} and \ref{mainconj2} hold with
\begin{itemize}
\item $\Lambda_7=\text{the set of weights of }L(\omega_7)$,
\item $\ell_7=241$.
\end{itemize}
Note that Conjecture \ref{mainconj2} implies
$$
\sum_{k=0}^{241}(-1)^k C_k^{(7)} D^k=\prod_{\lambda \in \Lambda_7}(1-e^{\lambda}D)
$$
and in particular,
$$
C_1^{(7)}=\sum_{\lambda \in \Lambda_7}e^{\lambda}=\chi(L(\omega_7))-7=q_7-8.
$$

\subsection{Type $F_4$}
We have $\dim L(\omega_1)=52$ and $\dim L(\omega_4)=26$. In order to describe $\Lambda_a$ and $\Lambda'_a$ for $a=1,4$, we first check that all non-trivial terms in the monomial expansions of $\chi(L(\omega_1))-\chi(L(\omega_4))-1$ and $\chi(L(\omega_4))-2$ as elements of $\mathbb{Z}[y^{\pm}_{a}]_{a\in I}$ have 1 as their coefficients. Thus we can define sets $\Lambda_1$ and $\Lambda_4$ of weights such that
$$
\sum_{\lambda\in \Lambda_1}e^{\lambda}=\chi(L(\omega_1))-\chi(L(\omega_4))-1
$$
and
$$
\sum_{\lambda\in \Lambda_4}e^{\lambda}=\chi(L(\omega_4))-2.
$$

For $a=1$, we expect that Conjectures \ref{mainconj} and \ref{mainconj2} hold with
\begin{itemize}
\item $\Lambda_1$ as above,
\item $\ell_1=25$.
\end{itemize}
Conjecture \ref{mainconj2} implies
$$
\sum_{k=0}^{25}(-1)^kC_k^{(1)}D^k=\prod_{\lambda\in \Lambda_1}(1-e^{\lambda}D)
$$
and in particular,
$$C_1^{(1)}=\sum_{\lambda\in \Lambda_1}e^{\lambda}=\chi(L(\omega_1))-\chi(L(\omega_4))-1=q_1-q_4-2.$$ 

For $a=4$, 
we expect that Conjectures \ref{mainconj} and \ref{mainconj2} hold with
\begin{itemize}
\item $\Lambda_4$ as above,
\item $\Lambda'_4=\Lambda_1$,
\item $\ell_4=74$.
\end{itemize}
Conjecture \ref{mainconj2} implies
$$
\sum_{k=0}^{74}(-1)^kC_k^{(1)}D^k=\prod_{\lambda\in \Lambda_2}(1-e^{\lambda}D)\prod_{\lambda\in \Lambda'_2}(1-e^{\lambda}D^2)
$$
and in particular,
$$
C_1^{(4)}=\chi(L(\omega_4))-2=q_4-2.
$$

\subsection{Type $G_2$}
Note that $\dim L(\omega_1)=14$ and $\dim L(\omega_2)=7$.

For $a=1$, we expect that Conjectures \ref{mainconj} and \ref{mainconj2} hold with
\begin{itemize}
\item $\Lambda_1=\{\omega_1,-\omega_1,\omega_1-3\omega_2,-\omega_1+3\omega_2,2\omega_1-3\omega_2,-2\omega_1+3\omega_2,0\}$,
\item $\ell_1=7$.
\end{itemize}
Conjecture \ref{mainconj2} implies
$$
\sum_{k=0}^{7}(-1)^k C_k^{(1)} D^k=\prod_{\lambda \in \Lambda_1}(1-e^{\lambda}D)
$$
and in particular,
$$
C_1^{(1)}=\sum_{\lambda \in \Lambda_1}e^{\lambda}=\chi\left(L(\omega_1) \right)-\chi\left( L(\omega_2) \right)=q_1-q_2-1.
$$
We also have
$$
\sum_{m=0}^{\infty}Q_{m}^{(1)}D^m=\frac{1+c_1D+c_2D^2+D^3}{\prod_{\lambda \in \Lambda_1}(1-e^{\lambda}D)}
$$
where $c_1=c_2=\chi\left(L(\omega_2)\right)+1=q_2+1$. And again we check its consistency by computing the dimension generating function from it :
$$
\begin{aligned}
\sum_{m=0}^{\infty}(\dim W_{m}^{(1)})D^m&=\frac{1+8 D+8 D^2+D^3}{(1-D)^7}\\
&=1+15 D+92 D^2+365 D^3+1113 D^4+\cdots
\end{aligned}
$$
which shows a good agreement.

For $a=2$, we expect that Conjectures \ref{mainconj} and \ref{mainconj2} hold with
\begin{itemize}
\item $\Lambda_2=\{\omega_2,-\omega_2,\omega_1-\omega_2,-\omega_1+\omega_2,\omega_1-2\omega_2,-\omega_1+2\omega_2\}$,
\item $\Lambda'_2=\Lambda_1$,
\item $\ell_2=27$.
\end{itemize}
Conjecture \ref{mainconj2} implies
$$
\sum_{k=0}^{27}(-1)^kC_k^{(2)}D^k=\prod_{\lambda\in \Lambda_2}(1-e^{\lambda}D)\prod_{\lambda\in \Lambda'_2}(1-e^{\lambda}D^3)
$$
and in particular,
$$
C_1^{(2)}=\sum_{\lambda \in \Lambda_2}e^{\lambda}=\chi\left(L(\omega_2) \right)-1=q_2-1.
$$

\section*{Acknowledgements}
The author wishes to thank Atsuo Kuniba for pointing out references on difference $L$-operators.

\appendix
\section{On the order of linear recurrence relations in Conjecture \ref{mainconj}}\label{ordell}
In this section, we give tables for $\ell_a$ in Conjecture \ref{mainconj}. It is convenient to introduce two sets of binomial like coefficients $L_{m,n}$ and $M_{m,n}$ for $m,n\in \mathbb{Z}$ with $0\leq n \leq m$.

Let us define $L_{m,n}$ for $m,n\in \mathbb{Z},\, 0\leq n \leq m$ recursively by
$$
L_{m,n}=2L_{m-1,n-1}+L_{m-1,n}
$$
with $L_{m,0}=1$ and $L_{m,m}=(3^m + 1)/2$. If we arrange them like Pascal's triangle, we obtain the following array :
$$
\begin{array}{cccccccccccc}
&    &    &    &    &  1\\
&    &    &    &  1 &    &  2\\
&    &    &  1 &    &  4 &    &  5\\
&    &  1 &    &  6 &    &  13 &    &  14\\
&  1 &    &  8 &    &  25 &    &  40 &    &  41\\
\cdots &    &  \cdots  &    & \cdots &     & \cdots  &    & \cdots & & \cdots 
\end{array}
$$

Similarly we define $M_{m,n}$ for $m,n\in \mathbb{Z},\, 0\leq n \leq m$ by
$$
M_{m,n}=2M_{m-1,n-1}+M_{m-1,n}
$$
with $M_{m,0}=1$ and $M_{m,m}=2\cdot 3^m-2^m$. If we arrange them like Pascal's triangle, we obtain the following array :
$$
\begin{array}{cccccccccccc}
&    &    &    &    &  1\\
&    &    &    &  1 &    &  4\\
&    &    &  1 &    &  6 &    &  14\\
&    &  1 &    &  8 &    &  26 &    &  46\\
&  1 &    &  10 &    &  42 &    &  98 &    &  146\\
\cdots &    &  \cdots  &    & \cdots &     & \cdots  &    & \cdots & & \cdots 
\end{array}
$$

The minimal order $\ell_a$ in Conjecture \ref{mainconj} is given as follows :

In type $A_r, \quad r\geq 1$, 
$$\ell_a=\binom{r+1}{a} \text{ for $a=1,\cdots,r$}.$$

In type $B_r, \quad r\geq 2$, 
$$
 \ell_a=
\begin{cases} 
 L_{r,a}, & \text{if $a=1,\cdots,r-1$}\\ 
 3^r-2^r+1, & \text{if $a=r$}.
\end{cases}
$$

In type $C_r, \quad r\geq 2$, 
$$
 \ell_a=
\begin{cases} 
 M_{r,a}, & \text{if $a=1,\cdots,r-1$}\\ 
 2^{r}, & \text{if $a=r$}.
\end{cases}
$$

In type $D_r, \quad r\geq 3$, 
$$
 \ell_a=
\begin{cases} 
 L_{r,a}, & \text{if $a=1,\cdots,r-2$}\\ 
 2^{r-1}, & \text{if $a=r-1,r$}.
\end{cases}
$$

In exceptional types,
$$
\begin{array}{c|l|}
\text{type} & (\ell_1,\cdots, \ell_r) \\
\hline
 E_6 & (27, 243, *, 243, 27, 73) \\
 E_7 & (127, *, *, *, *, 56, *) \\
 E_8 & (*,*,*,*,*,*,241,*) \\
 F_4 & (25, *, *, 74) \\
 G_2 & (7, 27)
\end{array}
$$
where the asterisk (*) means that the corresponding $\ell_a$ has not been yet computed successfully. We hope to complete this table in the near future. For readers' convenience we include the tables for classical types up to rank 7 :

In type $A_r, \quad r\geq 1$,
$$
\begin{array}{c|cccccccccc}
 r & \ell_1 & \cdots \\
 \hline
 1 & 2 & \text{} & \text{} & \text{} & \text{} & \text{} & \text{} \\
 2 & 3 & 3 & \text{} & \text{} & \text{} & \text{} & \text{} \\
 3 & 4 & 6 & 4 & \text{} & \text{} & \text{} & \text{} \\
 4 & 5 & 10 & 10 & 5 & \text{} & \text{} & \text{} \\
 5 & 6 & 15 & 20 & 15 & 6 & \text{} & \text{} \\
 6 & 7 & 21 & 35 & 35 & 21 & 7 & \text{} \\
 7 & 8 & 28 & 56 & 70 & 56 & 28 & 8 \\
\end{array}
$$

In type $B_r, \quad r\geq 2$, 
$$
\begin{array}{c|ccccccc}
 r & \ell_1 & \ell_2 & \cdots \\
 \hline
 2 & 4 & 6 & \text{} & \text{} & \text{} & \text{} & \text{} \\
 3 & 6 & 13 & 20 & \text{} & \text{} & \text{} & \text{} \\
 4 & 8 & 25 & 40 & 66 & \text{} & \text{} & \text{} \\
 5 & 10 & 41 & 90 & 121 & 212 & \text{} & \text{} \\
 6 & 12 & 61 & 172 & 301 & 364 & 666 & \text{} \\
 7 & 14 & 85 & 294 & 645 & 966 & 1093 & 2060 \\
\end{array}
$$

In type $C_r, \quad r\geq 2$, 
$$
\begin{array}{c|ccccccc}
 r & \ell_1 & \ell_2 & \cdots \\
 \hline
 2 & 6 & 4 & \text{} & \text{} & \text{} & \text{} & \text{} \\
 3 & 8 & 26 & 8 & \text{} & \text{} & \text{} & \text{} \\
 4 & 10 & 42 & 98 & 16 & \text{} & \text{} & \text{} \\
 5 & 12 & 62 & 182 & 342 & 32 & \text{} & \text{} \\
 6 & 14 & 86 & 306 & 706 & 1138 & 64 & \text{} \\
 7 & 16 & 114 & 478 & 1318 & 2550 & 3670 & 128 \\
\end{array}
$$

In type $D_r, \quad r\geq 3$, 
$$
\begin{array}{c|cccccccc}
 r & \ell_1 & \ell_2 & \ell_3 & \cdots & \text{} & \text{} & \text{} \\
 \hline
 3 & 6 & 4 & 4 & \text{} & \text{} & \text{} & \text{} \\
 4 & 8 & 25 & 8 & 8 & \text{} & \text{} & \text{} \\
 5 & 10 & 41 & 90 & 16 & 16 & \text{} & \text{} \\
 6 & 12 & 61 & 172 & 301 & 32 & 32 & \text{} \\
 7 & 14 & 85 & 294 & 645 & 966 & 64 & 64 \\
\end{array}
$$

\section{On the growth of the dimensions of the Kirillov-Reshetikhin modules}\label{dims}
Let us define a function $p_a:\mathbb{Z}_{\geq 0}\to \mathbb{Z}_{\geq 0}$ by $p_a(m)=\dim W_m^{(a)}$. Conjectures \ref{mainconj} and \ref{mainconj2} imply that $p_a$ is of polynomial growth in the sense that there exists a nonzero limit
$$
\lim_{m\to \infty}\frac{p_a(m)}{m^{e_a}}
$$
for some $e_a\in \mathbb{Z}_{\geq 0}$. This is argued in Theorem 5, \cite{MR1436775} for the simply-laced case. When $t_a=1$, $p_a$ is actually a polynomial. Even when $t_a\neq 1$, $p_a(m)$ can be expressed as a linear combination of terms of the form $m^n\zeta^m,\,n\in \mathbb{Z}_{\geq 0}$ where $\zeta$ denotes a $t_a$-th root of unity. For example, in type $G_2$, we have
$$
p_2(m)=\frac{(6+m) \left(a(m)+b(m)\cos (\frac{2 \pi m}{3})+c(m)\sin (\frac{2 \pi m}{3})\right)}{94478400}
$$
where 
$$
\begin{aligned}
a(m)&=3(4 + m) (5 + m) (6 + m) (7 + m) (8 + m) (715 + 948 m + 367 m^2 + 
   48 m^3 + 2 m^4),\\
b(m)&=240(6 + m) (25 + 12 m + m^2) (37 + 12 m + m^2),\\
c(n)&=160\sqrt{3}(3875 + 2592 m + 648 m^2 + 72 m^3 + 3 m^4).
\end{aligned}
$$

Let us call $e_a$ the degree of $p_a$ and denote it by $\deg p_a$.
Then the degree of $p_a$ is given as follows : 

In type $A_r, \quad r\geq 1$, 
$$\deg p_a=a(r+1-a) \text{ for $a=1,\cdots,r$}.$$

In Type $B_r, \quad r\geq 2$, 
$$
\deg p_a=
\begin{cases} 
 a(2r-a), & \text{if $a=1,\cdots,r-1$}\\ 
 r^2, & \text{if $a=r$}.
\end{cases}
$$

In type $C_r, \quad r\geq 2$, 
$$
\deg p_a=
\begin{cases} 
 a(2r+1-a), & \text{if $a=1,\cdots,r-1$}\\ 
 r(r+1)/2, & \text{if $a=r$}.
\end{cases}
$$

In type $D_r, \quad r\geq 3$, 
$$
\deg p_a=
\begin{cases} 
 a(2r-1-a), & \text{if $a=1,\cdots,r-2$}\\ 
 r(r-1)/2, & \text{if $a=r-1,r$}.
\end{cases}
$$

In exceptional types,
$$
\begin{array}{c|l|}
\text{type} & (\deg p_1,\cdots, \deg p_r) \\
\hline
 E_6 & (16,30,42,30,16,22) \\
 E_7 & (34,66,96,75,52,27,49) \\
 E_8 & (92,182,270,220,168,114,58,136) \\
 F_4 & (16,30,42,22) \\
 G_2 & (6,10)
\end{array}
$$

Having obtained these tables, we found that each $\deg p_a$ is simply given by
$$
\deg p_a=2\sum_{b\in I}(C^{-1})_{ab}.
$$
For example, the Cartan matrix $C$ for type $F_4$ is
$$
C=
\left(
\begin{array}{cccc}
 2 & -1 & 0 & 0 \\
 -1 & 2 & -1 & 0 \\
 0 & -2 & 2 & -1 \\
 0 & 0 & -1 & 2 \\
\end{array}
\right)
$$
and twice its inverse is
$$
2C^{-1}=
\left(
\begin{array}{cccc}
 4 & 6 & 4 & 2 \\
 6 & 12 & 8 & 4 \\
 8 & 16 & 12 & 6 \\
 4 & 8 & 6 & 4 \\
\end{array}
\right).
$$
The sum of the entries in each row is 16,30,42 and 22. 
We can explain this coincidence as follows :
\begin{proposition}
Let $q_{a}:\mathbb{Z}_{\geq 0}\to \mathbb{R}$ be a function of polynomial growth of degree $\deg q_a\in \mathbb{Z}_{\geq 0}$ for each $a\in I$. Suppose that they satisfy the $Q$-system : 
$$
q_{a}(m)^2 = q_{a}({m+1})q_{a}({m-1}) + \prod_{b : C_{ab}\neq 0} \prod_{k=0}^{-C_{a b}-1}q_{b}\left(\lfloor\frac{C_{b a}m - k}{C_{a b}}\rfloor\right)
$$
for each $a\in I, m \ge 1$. Then $\deg q_a=2\sum_{b\in I}(C^{-1})_{ab}$.
\begin{proof}
We can easily check that $q_a(m)^2-q_a(m-1)q_a(m+1)$ grows like $C\cdot m^{2\deg q_a-2}$ as $m\to \infty$ for some non-zero constant $C$. Hence $(\deg q_a)_{a\in I}$ must satisfy the following system of equations
$$
2\deg q_{a}-2=\sum_{b\in I}-C_{ab}\deg q_{b},\quad a\in I
$$
whose unique solution is
$$
\deg q_a=2\sum_{b\in I}(C^{-1})_{ab}.
$$
\end{proof}
\end{proposition}
\bibliographystyle{amsalpha}
\bibliography{linQ}
\end{document}